\documentclass[a4paper]{amsart}
\usepackage[utf8x]{inputenc}

\usepackage{fullpage}

\usepackage{amsmath}
\usepackage{amssymb}
\usepackage{latexsym}
\usepackage{amsxtra}
\usepackage{amscd}

\usepackage{graphics}
\usepackage{epic}
\usepackage{url}

\newtheorem{thm}{Theorem.}[section]

\newtheorem{prop}[thm]{Proposition.}

\newtheorem{rem}[thm]{Remark.}

\renewcommand{\em}{\sl}

\makeatletter
\renewcommand{\subsection}{\@startsection{subsection}{2}%
{\z@}{-3.25ex plus -1ex minus-.2ex}{-1em}{\bf}} \makeatother

\newcommand{\PP}{\mathbb{P}}
\newcommand{\ZZ}{\mathbb{Z}}
\newcommand{\CC}{\mathbb{C}}

\newcommand{\QQ}{\mathbb{Q}}

\renewcommand{\AA}{\mathbb{A}}

\newcommand{\A}{\mathcal{A}}
\newcommand{\T}{\mathcal{T}}

\newcommand{\K}{\mathcal{K}}
\renewcommand{\L}{\mathcal{L}}

\newcommand{\tr}{{\rm t}}

\newcommand{\MT}{{\rm MT}}

\newcommand{\GL}{{\rm GL}}
\newcommand{\SL}{{\rm SL}}

\newcommand{\MC}{{\rm MC}}

\newcommand{\Mat}{{\rm Mat}}

\newcommand{\diag}{{\rm diag}}

\newcommand{\SO}{{{\rm SO}}}

\newcommand{\Sp}{{\rm Sp}}

\newcommand{\rk}{{\rm rk}}
\newcommand{\id}{{\rm id}}
\newcommand{\J}{{\rm J}}

\newcommand{\la}{{\lambda}}
\renewcommand{\char}{{\rm char}}

\newcommand{\R}{{\mathcal R}}

\newcommand{\B}{{\mathcal B}}

\numberwithin{equation}{section}
\numberwithin{table}{section}
\numberwithin{thm}{section}
\theoremstyle{plain}

\begin{document}

\title{A differential equation with monodromy group $2.J_2$}

\author{Stefan Reiter}
\address[S.~Reiter]{
Department of Mathematics\\
University of Bayreuth\\
95440 Bayreuth\\
Germany}
\email{stefan.reiter@uni-bayreuth.de}

\begin{abstract}
We construct a sixth order differential equation having the central extension of $C_2$ by the Hall-Janko group $J_2$ as monodromy group.
Moreover it arises  from an iterated application of tensor products and convolution operations from a first order differential equation.
\end{abstract}

\maketitle

\section{Introduction}

According to \cite[5.6.1]{Wilson} there are two important constructions of the sporadic simple Hall-Janko group $J_2$ of order $604800$,
namely as permutation group on 100 points by Marshall Hall and as a quaternionic reflection group in 3 dimensions
in connection with the Leech lattice.
It is well known that $2.J_2$, the 
central extension of $C_2$ by the
Hall-Janko group $J_2$,
is an irreducible subgroup of $\Sp_6(\CC)$ \cite[p. 42-43]{Atlas}.
Generators of the six dimensional representation of $2.J_2$ over $\QQ(\zeta_5)$
were already
determined by Lindsey II \cite{LindseyII}.

Here we show that the group $2.J_2$ appears as a monodromy group of a sixth order differential equation
that can be constructed by an iterated application of tensor products and convolution operations
from a first order differential equation.

Throughout the article,  let $\partial=\frac{d}{dx}, \vartheta=x \partial$ and 
 $L=\sum_{i=0}^n a_i(x) \partial^i  \in \CC(x)[\partial]$ be a differential 
operator. Recall that 
 the {\it adjoint} $L^*$ of $L$  is 
 defined as  $L^*=\sum_{i=0}^n (-\partial)^i  a_i(x) $ and that $L$ is called {\it self adjoint} if 
 $L=(-1)^n L^*.$  If $L$ is self adjoint then the differential Galois group of $ L,$ 
 and hence the monodromy group of $L,$ is 
 contained in the symplectic group $\Sp_n(\CC)$ if $n$ is even \cite{Katz90}.

\begin{thm}\label{L2J_2}
The  formally self adjoint fuchsian operator 
\begin{eqnarray*}
L_{2.J_2}&=&250000\, \left( 6\,\vartheta+5 \right)  \left( 6\,\vartheta-1 \right)  \left( 3\,\vartheta-1
 \right)  \left( 3\,\vartheta+1 \right)  \left( 6\,\vartheta+1 \right)  \left( 6\,\vartheta-5
 \right)- \\
&&125\,x \left( 6\,\vartheta+1 \right)  \left( 6\,\vartheta+5 \right) \cdot \\
&&\left( 
1296000\,{\vartheta}^{4}+2592000\,{\vartheta}^{3}+2578320\,{\vartheta}^{2}+1282320\,\vartheta+213703
 \right)+ \\
&&11664\,x^2 \left( 10\,\vartheta+17 \right)  \left( 5\,\vartheta+7 \right) 
 \left( 10\,\vartheta+11 \right)  \left( 10\,\vartheta+9 \right)  \left( 5\,\vartheta+3
 \right)  \left( 10\,\vartheta+3 \right)
\end{eqnarray*}
has
the Riemann scheme
\[ \R(L_{2.J_2})=\left\{ \begin{array}{ccc}
                     0 & 1 & \infty \\
                  \hline
                   5/6 &  3 & 17/10\\
                   1/3&   5/2&   7/5\\
                   1/6&   2   & 11/10\\ 
                -1/6&    1  & 9/10 \\
              -1/3&    1/2 & 3/5 \\
               -5/6&   0&   3/10
            \end{array}\right\}\] 
and $2.J_2$ as monodromy group.
Moreover, its  monodromy representation $\rho: \pi_1(\PP^1\setminus \{0,1,\infty\},x_0)  \to \Sp_6(\CC)$ 
is uniquely determined  by the local monodromy, i.e. the Jordan forms. 
\end{thm}

\noindent Note that if $P=\sum x^i p_i(\vartheta)$ is any differential operator then it follows 
from  the identities  $\vartheta x^i =x^i (\vartheta+i)$ and $\vartheta^{\ast}=-\vartheta-1$ in  the ring $\CC(x)[\vartheta]$ 
that $P^\ast= \sum x^i p_i(-\vartheta-i-1) .$ 
Hence we obtain for the above differential operator $L_{2.J_2}$ that 
$$(x^{-1}L_{2.J_2}) ^{\ast}= x^{-1} L_{2.J_2},$$
meaning that 
the operator $L_{2.J_2}$ is formally self adjoint.

Necessary conditions that a tuple of matrices $\T=(T_1,\ldots, T_{r+1}),\; T_1\cdots
T_{r+1}=1,\; \T \in \GL_n(\CC)^{r+1},$
appears as a monodromy tuple of an irreducible fuchsian differential equation
are given by the Scott formula \cite[Theorem 1]{Scott77}:
\begin{eqnarray*}
 \sum_{i=1}^{r+1} \rk(T_i-1)&\geq& 2n, \\
\sum_{i=1}^{r+1} \dim C_{\GL_n(\CC)}(T_i) &\leq & (r-1)n^2+2, 
\end{eqnarray*}
 where $\dim (C_{\GL_n(\CC)}(T_i))$ denotes the dimension of the centralizer of $T_i$ in 
 $\GL_n(\CC)$. 
On the other hand the algorithm of Katz  \cite{katz96} provides a 
simple tool to check whether  
to a given tuple of Jordan forms $(J(T_1),\ldots,J(T_{r+1})) \in \GL_n(\CC)^{r+1}$
 satisfying
\[ \sum_{i=1}^{r+1} \dim C_{\GL_n(\CC)}(J(T_i)) = (r-1)n^2+2 \]
there exists a corresponding irreducible monodromy tuple $\T$.
Such a tuple $\T$ is called linearly rigid and it is  uniquely determined by its tuple of Jordan forms up to simultaneous
conjugation in $\GL_n(\CC)$.
A look in the Atlas of finite groups \cite[p. 43]{Atlas} 
allows to read off the Jordan forms
of the elements of $2.J_2$ in their six-dimensional representation. 
However applying the Scott formula and the
Katz-algorithm it turns out
that there is no linearly rigid irreducible tuple
that generates a subgroup of $2.J_2\subset \Sp_6(\CC)$.

One can weaken the notion of linear rigidity in the following way.
We say that an irreducible monodromy tuple $\T \in \Sp_n(\CC)^{r+1}$ is symplectically rigid
if the  {\rm dimension formula} holds:
\begin{equation}\label{orthrigid}  \sum_{i=1}^{r+1}
{\rm dim}(C_{\Sp_n(\CC)}(T_i))=(r-1)\dim (\Sp_n(\CC)),
\end{equation}
where $C_{\Sp_n(\CC)}(T_i)$ denotes the 
centralizer of local monodromy generator $T_i$ in the algebraic
group $\Sp_n(\CC).$ 
Symplectic rigidity is also a necessary condition for the existence of
only finitely many equivalence classes of irreducible tuples in the
symplectic group 
with given local monodromy \cite[Corollary~3.2]{SV}. 

In the six dimensional representation of  $2.J_2$
we find elements of order $6, 2, 10$ with Jordan forms
\[
(\zeta_6,\zeta_6,\zeta_3,\zeta_3^{-1},\zeta_6^{-1},\zeta_6^{-1}),\quad
(-1,-1,1,1,1,1), \quad
 (\zeta_{10},\zeta_{10}^3,\zeta_5^{3},\zeta_{5}^{-3},\zeta_{10}^{-3},\zeta_{10}^{-1}),\]
where $\zeta_k$ denotes a primitive $k$-th root of unity and
$(\zeta_6,\zeta_6,\zeta_3,\zeta_3^{-1},\zeta_6^{-1},\zeta_6^{-1})$ denotes an element having
Jordan form with two Jordan blocks of size one with  eigenvalues $\zeta_6$, $\zeta_6^{-1}$ resp., and
one Jordan block of size one with eigenvalues $\zeta_3$, $\zeta_3^{-1}$ resp..
Since the centralizer dimensions in $\Sp_6(\CC)$ of the elements are  $5, 13, 3$ resp.,
this triple satisfies the symplectic dimension formula.
The computation of the corresponding
normalized structure constant $n(\T)$, cf. \cite[Chapter I, Theorem~5.8]{MalleMatzat},
yields $n(\T)=1.$
Hence, such a  triple $\T$ in $(2.J_2)^3$ with these Jordan forms and product $1$
exists.
Moreover, this triple in $\Sp_6(\CC)^3$ obviously generates an irreducible subgroup of $2.J_2$,
since there is no invariant subspaces of dimension $1$, $2$ or $3$.  
Thus $\T$ is uniquely determined  in $(2.J_2)^3$ up to simultaneous conjugation.
Since the finite linear quasi-primitive groups generated by bi-reflections have been classified
in \cite{Wales78} it turns out that the generated group is $2.J_2$.

It remains to construct the differential operator $L_{2.J_2}$. 
For this we 
apply the Katz-Existence Algorithm to the monodromy triple $\T$.
Then we end up with a symplectically
rigid triple in dimension $4$ containing a bi-reflection.
This arises from a monodromy triple of a hypergeometric differential equation of order $4$
by taking the wedge product and applying a suitable middle convolution as shown in \cite[Theorem~3.3]{BR2012}.
Thus we are in the linearly rigid (hypergeometric) case.
Therefore this triple $\T \in \Sp_6(\CC)^3$ 
is uniquely determined by its triple of Jordan forms and
can be explicitly constructed from a rank one triple
together with the differential equation using an iterated sequence of tensor products
and convolution operations, see Section~\ref{Sec2}.

\section{Middle convolution}\label{Sec2}

We review some of the properties of the middle convolution
for monodromy tuples and differential operators, cf. \cite[Section 2]{DRFuchsian} and \cite{BR2012}.

The following result is a consequence of the numerology of the middle convolution $\MC_\la$
(cf.~\cite{katz96}):

\begin{prop} \label{lemmonodromy1} Let  $\A=(A_1,\dots,A_{r+1}) \in \GL_n(\CC)^{r+1}$
be an irreducible monodromy tuple of rank $n$
with at least two non trivial elements.
Further let  $\lambda\in \CC^\times \setminus \{1\}.$
Let $(\tilde{B}_1,\dots,\tilde{B}_{r+1})$
be the monodromy tuple $\MC_\la(\A).$ 
Then the following hold:
\begin{enumerate}
\item The monodromy tuple  $\MC_\la(\A)$ is again irreducible of rank
      \[ \rk(\MC_\la(\A))=\sum_{i=1}^r \rk(A_i-1)+\rk(A_{r+1}\lambda^{-1}-1)-\rk(\A).\]

\item Every  Jordan
block $\J(\alpha,l)$ occurring in the Jordan
decomposition of $A_i$ contributes a Jordan block $\J(\alpha
\lambda,l')$ to the Jordan decomposition of $\tilde{B}_i,$ where
$$ l':\;=\quad
  \begin{cases}
    \quad l, &
                              \quad\text{\rm if $\alpha \not= 1,\lambda^{-1}
$,} \\
    \quad  l-1& \quad \text{\rm if $\alpha =1$,} \\
    \quad l+1, & \quad \text{\rm if $\alpha =\lambda^{-1}$.}
  \end{cases}
  $$
  The only other Jordan blocks which occur in the Jordan
  decomposition of 
  $\tilde{B}_i$ are blocks of the form $\J(1,1).$

\item
Every  Jordan block $\J(\alpha^{-1},l)$ occurring in the Jordan
decomposition of  $A_{r+1}$ contributes a Jordan block $\J(\alpha^{-1}
\lambda^{-1},l')$ to the Jordan decomposition of $\tilde{B}_{r+1},$
where
$$ l':\;=\quad
  \begin{cases}
    \quad l, &
                              \quad\text{\rm if $\alpha \not= 1,\lambda^{-1}
$,} \\
    \quad  l+1& \quad \text{\rm if $\alpha =1$,} \\
    \quad l-1, & \quad \text{\rm if $\alpha =\lambda^{-1}$.}
  \end{cases}
  $$
  The only other Jordan blocks which occur in the Jordan
  decomposition of $\tilde{B}_{r+1}$ are blocks of the form $\J(\lambda^{-1},1).$
\item $MC_\la$ preserves linear rigidity.
\end{enumerate}\end{prop}

We have the following explicit construction over any field $K$ for $\MC_\la(\A)$  in Proposition~\ref{lemmonodromy1}, cf. \cite{DRFuchsian}:

The convolution $C_\lambda(\A)$ of $\A$ with $\la\in K^{\ast}$ is given by the following $r+1$ tuple  of matrices
$\B$
in $\GL_{nr}(K)^{r+1},$
where $B_i-1_{rn}$ is a block-matrix that is zero outside the $i$-th block row, which given by 
\[ ((A_1-1)\la,(A_2-1)\la,\ldots,(A_{i-1}-1)\la,A_i \la-1,A_{i+1}-1,\ldots,A_r-1) \]
for $i=1,\ldots,r$.
Moreover $B_1\cdots B_r-\la$ is the block matrix
\[\diag(A_2 \cdots A_r,A_3 \cdots A_r,\ldots ,1)
      \cdot \lambda \cdot \left(\begin{array}{ccccc}
        A_1-1 & \dots A_r-1 \\
             \vdots & & \\
         A_1-1 & \dots A_r-1
      \end{array}\right). \] 
There are the following  invariant subspaces:
\[ \K=\oplus \ker(A_i-1),\quad \L =\cap_{i=1}^r \ker(B_i-1)=\ker(B_{r+1}-1).\]
If $\la \neq 1 $ then $\L=\{ \diag(A_2\cdots A_r v,A_3\cdots A_r v,\dots,v)\mid v \in \ker(A_1\cdots A_r\la-1)\}$ and
 $\K\cap \L =0$.
The tuple $\MC_\la(\A)$  corresponding to the middle convolution is 
given by the action of $\B$ on $K^{rn}/(\K+\L)$.
Furthermore if $U \leq K^n$ is an $\A$ invariant subspace then $U^r$ is $\B$ invariant.

Let $(\la_1,\ldots, \la_{r+1})\in K^{r+1}$ be monodromy tuple of rank one.
Then we denote by
\[ \MT_{(\la_1,\ldots, \la_{r+1})} (\A):=(\la_1 A_1,\ldots,\la_{r+1} A_{r+1}).\]

Let  $\la_1,\la_2 \in K^{\ast}, \;\la=\la_1 \la_2,$ and
\begin{eqnarray*}
  F&:=&\MT_{(\lambda_2^{-1},1,\ldots,1,\la_2)}\circ C_\lambda \circ \MT_{(\lambda_1^{-1}\lambda_2 ,1,\ldots,1,\la_1 \la_2^{-1})} \circ C_{\lambda^{-1}} \circ \MT_{(\lambda_1,1,\ldots,1,\la_1^{-1})}, \\ 
\tilde{F}&:=&\MT_{(\lambda_2^{-1},1,\ldots,1,\la_2)}\circ \MC_\lambda \circ \MT_{(\lambda_1^{-1}\lambda_2 ,1,\ldots,1,\la_1 \la_2^{-1})} \circ \MC_{\lambda^{-1}} \circ \MT_{(\lambda_1,1,\ldots,1,\la_1^{-1})}. 
\end{eqnarray*}

In \cite[Cor. 5.15 a]{dr00} it is already shown that 
 $\tilde{F}$ 
preserves autoduality, i.e. if $\A$ is contained in a symplectic or an orthogonal group then the same holds for
$\tilde{F}(A)$.
But here we show the refined statement that  $\tilde{F}$ preserves a symmetric bilinear form, resp. an antisymmetric bilinear form.

\begin{thm}\label{form}
 Let $\A \in \GL_n(K)^{r+1}$ be an irreducible monodromy tuple, such that $A_i^{\tr} X A_i=X,\;i=1,\ldots,r,$ for some $0\neq X\in \Mat_n(K)$.
 \begin{enumerate}
  \item Then $\det(X) \neq 0$ and $X^t =X, X^t=-X$ resp.. 
 \item 
 Let $ \B:=F(\A).$
 Then there exists $0\neq Y\in \Mat_{nr^2}(K),\;Y^t= Y, Y^t=-Y  $ resp., such that
 \begin{eqnarray*}
 B_i^{\tr} Y B_i=Y, &&i=1,\ldots,r.
 \end{eqnarray*}
 Moreover the matrix $Y$ is defined via block-matrices as follows:
 
 \[ Y= D_1^{\tr} Y_0 D_1, \]
 where
  \begin{eqnarray*}
   D_1&=&\diag(D_{11},\ldots,D_{1r}), \\
    D_{11}&=& \diag(A_1 \la_1-1,A_2-1,\ldots,A_r-1), \\
    D_{1i}&=&\diag ((A_1 \la_1-1)\la^{-1},(A_2-1)\la^{-1},\ldots,(A_{i-1}-1)\la^{-1},A_i\la^{-1}-1,A_{i+1}-1,\ldots,A_r-1), \\
     &&\quad i>1,\\
  Y_0&=&(y_{ij})_{i,j=1,\ldots,r},  \\
   {y}_{11}&=& D_{21}^{\tr} H D_{21},\\
  {y}_{1j}&=&D_{31}^{\tr}  H  D_{2j} + D_{4j}^{\tr} H  (1-\la_1/\la_2),\quad 1<j, \\
   {y}_{j1}&=&D_{2j}^{\tr}  H  D_{31} +  H  D_{4j}(1-\la_1/\la_2), \quad 1<j, \\
  {y}_{ij}&=&D_{2i}^{\tr}  H  D_{2j} \la_1/\la_2, \quad 2\leq i,j \leq r,\\
     D_{21}&=& \diag(A_1/ \la_2-1,A_1\la_1-1,\ldots,A_1\la_1-1),  \\
     D_{2i}&=&\diag (A_i-1,\ldots,A_i-1),\; i>1,\\
      D_{31}&=& \diag((A_1- \la_1)/\la_2,A_1 \la_1-1,\ldots,A_1 \la_1-1),\\
     D_{4i}&=&\diag (0,A_i^{-1}-1,\ldots,A_{i}^{-1}-1,A_{i}^{-1}\la-1,(A_{i}^{-1}-1)\la, \ldots,,(A_{i}^{-1}-1)\la),\; i>1,\\
   H&=&(h_{ij})_{i,j=1,\ldots,r}  \in \Mat_{rn}(K),\quad  h_{i,j}=X.
  \end{eqnarray*}
Further, if $\char(K)\neq 2$ then
 $\tilde{F}$
 preserves both symmetric and antisymmetric bilinear forms.
 \end{enumerate}
\end{thm}

\begin{proof}
 \begin{enumerate}
  \item 
Since $\ker(X)$ is $\A$ invariant the irreducibility gives $\det(X) \neq 0$.
 Further, the equality
 $A_i^tX^tA_i=X^t$ implies $X^t =\gamma X$  for some $\gamma \in \bar{K}.$
 Hence $X=(X^t)^t=\gamma^2X$ gives (i).
 \item The claim for $Y$ and $F(\A)$ is a straightforward computation.
 It remains to show the claim for $\tilde{F}(A)$.
 Let 
  \[ \tilde{\A}=\MT_{(\lambda_1^{-1}\la_2,1,\ldots,1,\la_1 \la_2^{-1})} \circ C_{\lambda^{-1}} \circ \MT_{(\lambda_1,1,\ldots,1,\la_1^{-1})} (\A) \]
 and  
 $U =\oplus_{i=1}^r (\ker( \tilde{A}_i-1)+\K+\L),$ where 
 \begin{eqnarray*}
  \K&=& \ker(A_1\la_1-1) \oplus \ker(A_2-1)\oplus \ldots \oplus \ker(A_r-1)  \\
   \L&=&\{ \diag(A_2\cdots A_r v,\dots,v)\mid v \in \ker(\la_1A_1\cdots A_r\la^{-1}-1)\}
 \end{eqnarray*}
 are the $\B$ invariant subspaces that arise in the convolution process.
 Further,
   \[ \ker(  \tilde{A}_1-1)=(\ker(A_1\la_1^{-1}-1),0,\ldots,0)\]
  is in the kernel of
  $D_{1j}D_{21}^{\tr} H D_{21}D_{11}$ since $(A_1\la_1-1)^\tr H=-H \la_1 A_1^{-1}(A_1\la_1^{-1}-1)$.
 Hence it is straightforward to check that $Y(U)=0$.
 Since $ \tilde{B}\cong \tilde{F}(A)$
 is the irreducible tuple induced by the action of $\B$ on 
 $K^{nr^2}/U$ the claim follows.
\end{enumerate} 
\end{proof}

By the Riemann-Hilbert correspondence, each monodromy tuple $\T\in \GL_n(\CC)^{r+1}$
corresponds to an ordinary Fuchsian differential equation (or, equivalently, an 
operator $L=\sum_{i=0}^mx^iP_i(\vartheta)\in\mathbb{C}[x,\vartheta]$ 
in the Weyl algebra $\CC[x,\vartheta=x\frac{d}{dx}]$) with regular singularities $x_1,\ldots,x_r,x_{r+1}=\infty$). 
Let  $f$ be a solution of $L,$  
viewed as a section of the local  system $\L$ of solutions of $L,$ and let 
 $a\in\QQ\setminus\ZZ$.
 For two simple loops $\gamma_p, \gamma_q,$ based at $x_0\in \AA^1\setminus \{x_1,\ldots,x_r\},$ and   
 moving counterclockwise around $p,$ resp. $q,$ we define the \textit{Pochhammer contour} \[[\gamma_p,\gamma_q]:=\gamma_p^{-1}\gamma_q^{-1}\gamma_p\gamma_q.\]
 For $y\in \AA^1\setminus \{x_1,\ldots,x_r\},$ the integral 
   \begin{equation}\label{int} C^{p}_a(f)(y):=\int_{[\gamma_{p},\gamma_{y}]}f(x)(y-x)^a\frac{dx}{y-x}
   \end{equation} is called the 
   \textit{convolution} of $f$ and $x^a$ with respect to the Pochhammer contour $[\gamma_{p},\gamma_{y}].$ 
In \cite[Prop. 4.10]{BR2012} it is shown that
$C^p_a(f)$ is a solution of
\begin{eqnarray}\label{CaL}
\mathcal{C}_a(L):=\sum_{i=0}^my^i\prod_{j=0}^{i-1}(\vartheta+i-a-j)\prod_{k=0}^{m-i-1}(\vartheta-k)P_i(\vartheta-a)
\in \mathbb{C}[y,\vartheta] 
\end{eqnarray}
for each $p\in\mathbb{P}^1$.
In general $\mathcal{C}_a(L)$ is not irreducible but
the
factor that coincides with the 
differential operator associated to the middle convolution $\MC_\la(\T), \; \la=\exp(2\pi i a),$ via 
the Riemann-Hilbert correspondence
can be often easily determined, cf.  \cite[Cor.~4.16]{BR2012}.

Note further that if $\T$ is a monodromy tuple of $L(\vartheta)$, where $T_1$ is the local monodromy at $0$,
then the tensor product $ \MT_{(\la,1,\ldots,1,\la^{-1})}(\T)$
changes
$L(\vartheta)$ to $L(\vartheta-a)$, $\exp(2\pi i a)=\lambda$. 

\section{Proof of Theorem~\ref{L2J_2}}

\begin{proof}

Let  $\T\in \Sp_6(\CC)^3$ be a monodromy triple with  Jordan forms
\[
(\zeta_6,\zeta_6,\zeta_3,\zeta_3^{-1},\zeta_6^{-1},\zeta_6^{-1}),\quad
(-1,-1,1,1,1,1), \quad
 (\zeta_{10},\zeta_{10}^3,\zeta_5^{3},\zeta_{5}^{-3},\zeta_{10}^{-3},\zeta_{10}^{-1}).\]
Applying the sequence
 \[ \MT_{(\zeta_{5}^{3},1,\zeta_{5}^{-3})} \circ  \MC_{\zeta_{5}^{-3}\zeta_6}  \circ    \MT_{(\zeta_{5}^{-3}\zeta_6^{-1},1,\zeta_6\zeta_5^{3})} \circ  
 \MC_{\zeta_{5}^{3}\zeta_6^{-1}}       \circ   \MT_{(\zeta_6,1,\zeta_6^{-1})}
\]
we get  we get a  monodromy triple $\tilde{\T} \in \Sp_4(\CC)^3$ with Jordan forms
\[  (\zeta_{5}^3,\zeta_3,\zeta_3^{-1},\zeta_{5}^{-3}),\quad (-1,-1,1,1),\quad
    (\zeta_{10}^1,\zeta_{10}^3,\zeta_{10}^{-3},\zeta_{10}^{-1})\]
 by Theorem~\ref{form} and Proposition~\ref{lemmonodromy1}.   
This triple is again symplectically rigid by (\ref{orthrigid}).
Due to \cite[Theorem~3.3]{BR2012} we know that $\tilde{\T}$ is uniquely determined by its Jordan forms.
This also shows the existence and  uniqueness of $\T$ since the middle convolution $\MC_\la$ is invertible, i.e.
$\MC_\la \circ \MC_{\la^{-1}} \cong \id$ by  \cite[Theorem~3.5 and Proposition~3.2]{dr00}.

Moreover it is also shown in \cite[Theorem~3.3]{BR2012} that $\tilde{T}$ arises from a monodromy triple of a hypergeometric differential operator $L_4$ of order $4$
by taking the wedge product and applying the middle convolution $MC_{-1}$.

This allows us to construct the operator $L_{2.J_2}$
by applying the corresponding operations for the fuchsian differential operators
as explained in \cite[Section 4]{BR2012} and indicated in Section~\ref{Sec2}.
We start with the hypergeometric differential operator
\begin{eqnarray*}
L_4&=&256\, \left( 15\,\vartheta-13 \right)  \left( 15\,\vartheta-7 \right)  \left( 15\,\vartheta-8 \right)  \left( 
15\,\vartheta-2 \right) \\
 &&-81 x \, \left( 20\,\vartheta-11 \right)  \left( 20\,\vartheta+13 \right)  \left( 20\,\vartheta-3
 \right)  \left( 20\,\vartheta+1 \right) 
\end{eqnarray*}
with Riemann scheme, cf. \cite[Section 2]{BH}, 
\[ \R(L_4)=\left\{ \begin{array}{ccc}
                     0 & 1 & \infty \\
                  \hline
                   2/15 &  0 & -11/20\\
                   7/15&   1&   -3/20\\
                   8/15&   1   & 1/20\\ 
                13/15&    2  & 13/20 \\
                        \end{array}\right\}.\] 
Since $\Lambda^2 (\SL_4)\cong \SO_6$ we get an operator of degree six with a monodromy group in $\SO_6(\CC)$ having  local monodromy 
\[ (\zeta_5^2,\zeta_3,1,1,\zeta_3^{-1},\zeta_5^{-2}),\quad (\J(2),\J(2),1,1),\quad
    (\zeta_{10},\zeta_{10}^3,-1,-1,\zeta_{10}^{-3},\zeta_{10}^{-1}).\]
One should note that the two differential operators, namely $L_4$ and its dual $L_4^\ast$,  give rise to  equivalent
operators of degree six.

Hence $\MC_{-1}$ yields a symplectically rigid triple of rank $4$ by
Proposition~\ref{lemmonodromy1}  and \cite[Cor. 5.15]{dr00}
with  local monodromy at $0, 1$ and $\infty$
\[  (\zeta_{10},\zeta_6,\zeta_6^{-1},\zeta_{10}^{-1}),\quad (-1,-1,1,1),\quad
    (\zeta_5,\zeta_5^2,\zeta_5^{-2},\zeta_5^{-1}).\]
The corresponding formally self adjoint operator is then
\begin{eqnarray*}
 P&=&900\, \left( 6\,\vartheta+5 \right)  \left( 10\,\vartheta+1 \right)  \left( 10\,\vartheta+9 \right)  \left( 6\,\vartheta+1 \right)\\
  && -x(6480000\,{\vartheta}^{4}+25920000\,{\vartheta}^{3}
+42051600\,{\vartheta}^{2}+32263200\,\vartheta+9522215)\\
&&+x^2 5184\, \left( 5\,\vartheta+11 \right)  \left( 5\,\vartheta+7 \right)  \left( 5\,\vartheta+8 \right)  \left( 5\,\vartheta+4 \right).
\end{eqnarray*}

Applying $\MT_{(-1,1,-1)}$ we get by uniqueness $\tilde{\T}$.
Further,  the inverse sequence of the beginning of the proof
\[ \MT_{(\zeta_{6}^{-1},1,\zeta_{6})} \circ  \MC_{\zeta_{5}^{-3}\zeta_6}  \circ    \MT_{(\zeta_{5}^{3}\zeta_6,1,\zeta_6^{-1}\zeta_5^{-3})} \circ  
 \MC_{\zeta_{5}^{3}\zeta_6^{-1}}       \circ   \MT_{(\zeta_5^{-3},1,\zeta_5^{3})}
\]

gives finally $\T$ and $L_{2.J_2}$.
Note that applying this sequence to $P$ gives $L_{2.J_2}$ as an irreducible factor of $L_3(\vartheta+1/6),$
where by $(\ref{CaL})$
\[ L_3(\vartheta):=C_{2/5+1/6} (L_2(\vartheta-3/5-1/6+2)),\quad L_2:= C_{3/5+5/6} (P(\vartheta-9/10)). \]
Thus
\begin{eqnarray*}
L_3(\vartheta)&=&6750000\,\vartheta \left( \vartheta-1 \right)  \left( 3\,\vartheta+2 \right)  \left( 3\,\vartheta-1 \right) \left( 6\,\vartheta+1 \right) \left( 2\,\vartheta-1 \right)  \left( 30\,\vartheta-17 \right)  \left( 30\,\vartheta+7
 \right)   \\
 &&-1125 x\,\vartheta \left( 3\,\vartheta+2 \right)  \left( 30\,\vartheta+37 \right)  \left( 30\,\vartheta+13 \right) \cdot \\
 &&\left( 432000\,{\vartheta}^{4}+576000
\,{\vartheta}^{3}+499440\,{\vartheta}^{2}+204960\,\vartheta+20201 \right)\\
&&+16x^2 \, \left( 15\,\vartheta+2 \right)  \left( 15\,\vartheta+23 \right)  \left( 30\,\vartheta+67 \right)  \left( 30\,\vartheta+
37 \right)  \left( 15\,\vartheta+14 \right)  \left( 15\,\vartheta+11 \right) \cdot \\
&& \left( 30\,\vartheta+43 \right)  \left( 30\,\vartheta+13 \right) 
\end{eqnarray*}
and we get the factorization
 \[ L_3(\vartheta) = (30 \vartheta - 17) (30 \vartheta + 7) \; L_{2.J_2}(\vartheta-1/6). \]
 
Since there is a triple  in $(2.J_2)^3$ with  product $1$ and same Jordan forms as  $\T$,
which can be verified by computing the corresponding
normalized structure constant, cf. \cite[Chap. I, Theorem~5.8]{MalleMatzat} and \cite[p. 43]{Atlas},
the uniqueness implies that the monodromy group of $L_{2.J_2}$ is a subgroup of $2.J_2$.
Since the finite linear quasi-primitive groups generated by bi-reflections have been classified
in \cite[Main Theorem]{Wales78} one gets that the generated group is $2.J_2$ since the trace of the element of order $10$ is in $\QQ(\zeta_5)$ but not in $\QQ$.
\end{proof}

Starting with the monodromy group generators of the 
generalized hypergeometric differential equation of order $4$, e.g. taken from
\cite[Theorem~3.5]{BH},
and performing the operations in the above proof on the level of
monodromy tuples, cf. Section~2 or \cite{dr00}, one can also construct the corresponding
monodromy tuple for $L_{2.J_2}$ and the invariant hermitian matrix $H$
over the ring of integers
$\ZZ[i,\frac{1+\sqrt{5}}{2}].$

\begin{rem}
 Starting with the hypergeometric operator
 \begin{eqnarray*}
   L_4&=&16 (2 \vartheta+2 a_1-c_1-1) (2 \vartheta-2 a_1+c_1-1) (2 \vartheta-2 a_1-c_1-1) (2 \vartheta+2 a_1+c_1-1)\\
     & &-x(4 \vartheta+2 (c_3+c_2)+1) (4 \vartheta+2 (c_2-c_3)-1) (4 \vartheta-2 (c_3+c_2)+1) (4 \vartheta+2 (c_3-c_2)-1)
 \end{eqnarray*}
we get analogously to the proof of Theorem~1.1 the formally adjoint operator of degree $6$

\begin{eqnarray*} 
 L&=&64 ( \vartheta-a_1) ( \vartheta+a_1) ( \vartheta-2 a_1) ( \vartheta+2 a_1) ( \vartheta+1+a_1) ( \vartheta-1-a_1)\\
   &&-x( \vartheta+1+a_1)( \vartheta-a_1) \cdot \\
   &&(128  \vartheta^4+256  \vartheta^3+  \vartheta^2(-64 v_1+304)+ \vartheta (-64v_1+176)-32 v_2+16 v_1^2-24 v_1+39)\\
&&+x^2 64 ( \vartheta+1+c_3) ( \vartheta+1-c_3) ( \vartheta+1+c_2) ( \vartheta+1-c_2) ( \vartheta+1+c_1) ( \vartheta+1-c_1),
\end{eqnarray*}
where
\begin{eqnarray*} 
 v_1=(2a_1)^2+c_1^2+c_2^2+c_3^2, && v_2=(2a_1)^4+c_1^4+c_2^4+c_3^4, 
 \end{eqnarray*}
with Riemann scheme
\[\R(L)=\left\{ \begin{array}{ccc}
                     0 & 1 & \infty \\
                  \hline
                   1+a_1 &  3 & c_3+1\\
                   2a_1&   5/2&  c_2+1\\
                   a_1&   2   & c_1+1\\ 
                -a_1&    1  & -c_1+1 \\
              -2a_1&    1/2 & -c_2+1 \\
               -1-a_1&   0&   -c_3+1
            \end{array}\right\}.\] 

Specializing the parameters $(a_1, c_1, c_2, c_3)$ we get further examples of operators having monodromy group $2.J_2$:

\[
 \begin{array}{cl|}
  a_1 & (c_1,c_2,c_3)  \\ 
  -1/6  & (19/20,9/20,3/4) \\
        & (13/20,3/20,3/4) \\
        & (9/10,7/10,3/5) \\
         & (9/10,7/10,4/5) \\
         &\\
  \end{array}  
  \begin{array}{cl|}
   a_1 & (c_1,c_2,c_3)  \\
    -1/5 & (11/12,7/12,3/4) \\
       & (6/7,5/7,4/7) \\
       & (17/20,13/20,3/4) \\
        & (7/12,1/12,2/3) \\
         & (14/15,11/15,2/3) \     
      \end{array}  
  \begin{array}{cl}    
   a_1 & (c_1,c_2,c_3)  \\
   -2/5 & (11/12,7/12,3/4) \\
       & (6/7,5/7,4/7) \\
       & (19/20,9/20,3/4) \\
        & (7/12,1/12,2/3) \\
         & (13/15,8/15,2/3) \\     
    \end{array}          
\]
We also get the following irreducible operators having finite monodromy group contained
in $2.J_2$, cf. \cite[p. 42]{Atlas}:
\[
 \begin{array}{ccccc}
  a_1 & (c_1,c_2,c_3) & \mbox{ group }\\
  -1/6 & (11/12,7/12,3/4) & 2^{3+4}:(3 \times S_3) \\
       & (13/14,11/14,9/14) & 2 \times U_3(3) \\
       & (7/8,5/8,3/4) &  2 \times U_3(3) 
\end{array}\]

The group    $2^{3+4}:(3 \times S_3)$ is an imprimitive subgroup of $2.J_2 \subseteq \Sp_6(\CC)$ of order $2^8 \cdot 3^2$.    
It is a transitive group on $24$ points and has the  transitive group identification number $5045$.
 \end{rem}


\begin{thebibliography}{10}

\bibitem{BH}
F.~Beukers and G.~Heckman.
\newblock Monodromy for the hypergeometric function $\sb n {F}\sb{n-1}$.
\newblock {\em Invent. Math.}, 95(2):\ 325--354, 1989.

\bibitem{BR2012}
M.~Bogner and S.~Reiter.
\newblock {On symplectically rigid local systems of rank four and Calabi-Yau
  operators}.
\newblock {\em J. Symb. Comput.}, 48:64--100, 2012.

\bibitem{Atlas}
J.~H. Conway, R.~T. Curtis, S.~P. Norton, R.~A. Parker, and R.~A. Wilson.
\newblock {\em Atlas of finite groups. Maximal subgroups and ordinary
  characters for simple groups. With computational assistance from J. G.
  Thackray}.
\newblock Oxford University Press, 1985.

\bibitem{dr00}
M.\ Dettweiler and S.\ Reiter.
\newblock An algorithm of {K}atz and its application to the inverse {G}alois
  problem.
\newblock {\em J. Symb. Comput.}, 30:\ 761--798, 2000.

\bibitem{DRFuchsian}
M.~Dettweiler and S.~Reiter.
\newblock {Middle convolution of Fuchsian systems and the construction of rigid
  differential systems}.
\newblock {\em Journal of Algebra}, 318:1--24, 2007.

\bibitem{LindseyII}
J.~H.~Lindsey II.
\newblock {A new lattice for the Hall-Janko group}.
\newblock {\em Proc. Amer. Math. Soc.}, 103:703--709, 1988.

\bibitem{Katz90}
N.M.\ Katz.
\newblock {\em Exponential sums and differential equations}.
\newblock Annals of Mathematics Studies 124. Princeton University Press, 1990.

\bibitem{katz96}
N.M.\ Katz.
\newblock {\em Rigid Local Systems}.
\newblock Annals of Mathematics Studies 139. Princeton University Press, 1996.

\bibitem{MalleMatzat}
G.\ Malle and B.H.\ Matzat.
\newblock {\em Inverse {G}alois Theory}.
\newblock Monographs in Mathematics. Springer-Verlag, 1999.

\bibitem{Scott77}
L.~L. Scott.
\newblock Matrices and cohomology.
\newblock {\em Ann. {M}ath.}, 105:473--492, 1977.

\bibitem{SV}
K.~Strambach and H.~V\"olklein.
\newblock On linearly rigid tuples.
\newblock {\em J. Reine Angew. Math.}, 510:57--62, 1999.

\bibitem{Wales78}
D.~B. Wales.
\newblock { Linear groups of degree n containing an involution with two
  eigenvalues −1. II}.
\newblock {\em J. Algebra}, 53(1):\ 58–67, 1978.

\bibitem{Wilson}
R.~A. Wilson.
\newblock {\em The finite simple groups}.
\newblock Number 251 in Graduate Texts in Mathematics. Springer-Verlag, 2009.

\end{thebibliography}
\end{document}